\documentclass[11pt,reqno]{amsart}
    \usepackage{amssymb}
    \usepackage{amsmath}
    \usepackage{mathrsfs}
    \usepackage{amssymb, amsmath, amsfonts,tikz}
    \usepackage[title]{appendix}
    \usepackage{mathabx}
   \usepackage{color}

    \usepackage[hidelinks]{hyperref}
    \allowdisplaybreaks
    \usepackage[numbers,sort&compress]{natbib}

    \makeatletter
    \def\tank#1{\mathbb Protected@xdef\@thanks{\@thanks
     \mathbb Protect\footnotetext[0]{#1}}}
    \def\bigfoot{

     \@footnotetext}
    \makeatother

    \newcommand{\ea}{\end{array}}

    \allowdisplaybreaks
    \numberwithin{equation}{section}

    \allowdisplaybreaks

    \newtheorem{theorem}{Theorem}[section]
    \newtheorem{lemma}{Lemma}[section]
    \newtheorem{proposition}[theorem]{Proposition}

    \newtheorem{corollary}[theorem]{Corollary}

    \def\beq{\begin{equation}}
    \def\nneq{\end{equation}}

    \def\bthm{\begin{theorem}}
    \def\nthm{\end{theorem}}

    \def\blem{\begin{lemma}}
    \def\nlem{\end{lemma}}
    \def\bprf{\begin{proof}}
    \def\nprf{\end{proof}}
    \def\bprop{\begin{prop}}
    \def\nprop{\end{prop}}
    \def\brmk{\begin{rem}}
    \def\nrmk{\end{rem}}

    \def\bexa{\begin{exa}}
    \def\nexa{\end{exa}}
    \def\bcor{\begin{cor}}
    \def\ncor{\end{cor}}

        \newcommand{\E}{\mathbb E}

      \def\R{\mathbb{R}}

    \title[Local linearization  for  the SWE]{Local linearization for  estimating the diffusion parameter of  nonlinear stochastic   wave equations with spatially correlated noise}

         \date{}
    
\begin{document}

            \author[G.P. Liu]{Guoping Liu}  
\address[]{Guoping Liu, School of Mathematics and Statistics, 
Huazhong University of Science and Technology,  Wuhan, 430074, 
China.}
\email{liuguoping@hust.edu.cn} 

    \author[R. Wang]{Ran Wang}
    \address[]{Ran Wang, School of Mathematics and Statistics,  Wuhan University,  Wuhan, 430072,
    China.}
    \email{rwang@whu.edu.cn}

    \maketitle
     \noindent {\bf Abstract:}  We study the bi-parameter local linearization  of the one-dimensional nonlinear stochastic wave equation  driven by a Gaussian noise,   which  is white in time and has a spatially homogeneous covariance structure of Riesz-kernel type.   We establish that the second-order increments of the solution can be approximated by those of the corresponding linearized wave equation, modulated by the diffusion coefficient. These findings extend the previous results of Huang et al. \cite{HOO2024}, which addressed the case of space-time white noise. As applications, we analyze the quadratic variation of the solution and construct a consistent estimator for the diffusion parameter.
 \vskip0.3cm

    \section{Introduction  }

Consider the non-linear stochastic wave equation (SWE, for short) in one spatial dimension:
      \begin{equation}\label{SWE}
\begin{cases}
\vspace{6pt}
\displaystyle{\frac{\partial^2}{\partial t^2} u(t, x) =\frac{\partial^2}{\partial x^2} u(t, x) +  F(u(t,x))\dot W(t, x), \quad t\ge 0, \, x \in \mathbb{R},}\\
\displaystyle{u(0, x) = 0, \quad \frac{\partial}{\partial t} u(0, x) =0.}
\end{cases}
\end{equation}
The diffusion coefficient   $F:\mathbb R\rightarrow\mathbb R$  is assumed to be  globally  Lipschitz continuous. The term  $\dot W$ is a centered  Gaussian field on a probability space $(\Omega, \mathcal F, \mathbb P)$, which is  white in time and 
 has a spatially homogeneous correlation of fractional type.   More precisely, the noise $\dot W$ is defined by a family of centered Gaussian random variables $\{W(\varphi), \, \varphi\in \mathcal D\}$, where $\mathcal D:=C_0^{\infty}([0,\infty\times \mathbb R))$ is the space of infinitely differentiable functions with compact support, with the covariance function  
\begin{align}
\begin{aligned}\label{Eq:noise_cov}
\E[W(\varphi) W(\psi)] 
& =  \int_{\R_+} dt \int_\R \mu(d\xi)\, \mathscr{F}(\varphi(t, \cdot))(\xi) \overline{\mathscr{F}(\psi(t, \cdot))(\xi)},
\end{aligned}
\end{align}
for all $\varphi, \psi \in\mathcal D$. Here,  $\mathscr{F}(\varphi(s, \cdot))(\xi)$ denotes the Fourier transform of the function $y \mapsto \varphi(s, y)$,  
\[ \mathscr{F}(\varphi(s, \cdot))(\xi): = \int_\R e^{-i\xi y} \varphi(s, y) dy.
\]
For any $H\in [1/2,1)$, the spectral measure $\mu_H$ is given by  
\begin{align}\label{eq constant}
\mu(d\xi) :=c_H |\xi|^{1-2H} d\xi, \ \   \  \ \text{with }\   c_H = \frac{\Gamma(2H+1)\sin (\pi H)}{2\pi}.
\end{align} 

  When $H=\frac12$, that is, $\dot W$ is the space-time Gaussian white noise,  the   existence  of
  a real-valued   solution to  \eqref{SWE} was studied in Walsh  \cite{Walsh86}.      When $H\in (\frac12,1)$,  the solution to   \eqref{SWE} was studied in  Dalang   \cite{D99} and the H\"older continuity was studied in Dalang and  Sanz-Sol\'e \cite{DS15}.   Huang et al. \cite{HOO2024} recently established a local linearization property for solutions to the  SWE  \eqref{SWE} with space-time white noise. Following this research, we generalize the approximation scheme developed in \cite{HOO2024} to the case of spatially colored noise in this paper.  
   
   To present the local linearization result, we first introduce the following stochastic heat equation (SHE, for short):
        \begin{equation}\label{SHE}
\begin{cases}
\vspace{6pt}
\displaystyle{\frac{\partial}{\partial t} X(t, x) =\frac{\partial^2}{\partial x^2} X(t, x) +   F(X(t,x)) \dot{W}(t, x), \quad t \ge 0, x \in \mathbb{R},}\\
\displaystyle{X(0, x) = 0.}
\end{cases}
\end{equation} 
It is well known that under  the Lipschitz continuity  of $F$, the solution to \eqref{SHE} exhibits a local  linearization property; see, e.g., \cite{FKM2015, H14}.  Specifically,   let    $Y$ denote the linearized version of $X$; that is,   $Y$ is the solution to the   SHE \eqref{SHE} with $F\equiv1$. Then, 
  \begin{equation}\label{SHE appro}
  X(t, x + \varepsilon) - X(t, x) = F(X(t, x)) \left\{Y(t, x + \varepsilon) - Y(t, x)\right\} + R_{\varepsilon}(t, x),
\end{equation}  
where  the remainder term $R_{\varepsilon}(t, x)$ tends to $0$   as $\varepsilon\to 0$ at a rate   faster than $Y(t, x + \varepsilon) -Y(t, x)$.   
The relation \eqref{SHE appro} implies that, for fixed $t>0$, the  local spatial  fluctuations  of    $X(t, \cdot)$ are essentially governed by those of  $Y(t, \cdot)$.   In other words,  ignoring precise regularity conditions,   \eqref{SHE appro} indicates  that $X(t, \cdot)$ is controlled by $Y(t, \cdot)$ in the sense of Gubinelli's theory of  controlled paths     \cite{G04}.  
  
A similar local linearization behavior holds for  temporal increments: for fixed $t>0$ and $x\in \mathbb R$, as $\varepsilon \downarrow 0$, the increment  $X(t+\varepsilon, x)-X(t, x)$ admits an analogous structure.  See \cite{KSXZ2013, HP15, WX2024, QWWX2025}. Those results provide a   quantitative framework for analyzing the local   structure   of   sample paths  of the solution, including properties such as     Khinchin's law of the iterated logarithm,  Chung's  law of the iterated logarithm,    quadratic variation  of the  process, and  small-ball probability estimates. See, e.g.,  \cite{CHKK19, DNP2025, Das2022, HK2017, KKM23}. 
  The local  linearization of spatio-temporal increments of   solutions  to nonlinear SHEs has been further  studied by  Hu and Lee in  \cite{HL2025}.

In \cite{HK2016}, Huang and Khoshnevisan  investigated  an analogous problem  for the  SWE with space-time white noise.  They showed that, in contrast to the case of the SHE,  the spatial increment of the  solution to  the SWE with initial data $(u_0, u_1) \equiv (0, 1)$ does not  exhibit local linearization  for fixed $t$.   Remarkably, by adopting a bi-parameter perspective, Huang  et al.   \cite{HOO2024}   revisited the local linearization problem  for  the SWE.   Specifically, they introduced  a new coordinate system $(\tau, \lambda)$ obtained by rotating the $(t, x)$-coordinates by $-45^\circ$.  That is, 
\begin{equation}\label{eq transform}
(\tau, \lambda): = \Big(\frac{t-x}{\sqrt{2}}, \frac{t+x}{\sqrt{2}}\Big) \quad \text{and} \quad (t, x) 
= \Big( \frac{\tau+\lambda}{\sqrt{2}}, 
\frac{-\tau+\lambda}{\sqrt{2}} \Big). 
\end{equation}
For $\tau>0$ and $\lambda\ge -\tau$,  define 
\begin{equation}\label{SWE 2}
v(\tau, \lambda) := u\left(\frac{\tau+\lambda}{\sqrt{2}}, \frac{-\tau+\lambda}{\sqrt{2}}\right). 
\end{equation}
  Let $U:=\{U(t,x)\}_{t\ge0, x\in \mathbb R}$ be the solution of the following  linear SWE in one spatial dimension:
      \begin{equation}\label{SWE linear}
\begin{cases}
\vspace{6pt}
\displaystyle{\frac{\partial^2}{\partial t^2}U(t, x) =\frac{\partial^2}{\partial x^2} U(t, x) +   \dot W(t, x), \quad t\ge 0, \, x \in \mathbb{R},}\\
\displaystyle{U(0, x) = 0, \quad \frac{\partial}{\partial t} U(0, x) = 0.}
\end{cases}
\end{equation}
 For any $\tau\ge 0$ and $\lambda\ge -\tau$,  denote 
\begin{equation}\label{SWE 3}
V(\tau, \lambda) := U\left(\frac{\tau+\lambda}{\sqrt{2}}, \frac{-\tau+\lambda}{\sqrt{2}}\right). 
\end{equation}
 For $\varepsilon \in \mathbb R$, define the difference operators $\delta_{\varepsilon}^{(j)}, j=1, 2$  by  
\begin{equation}\label{eq diff}
 \begin{split}
 \delta_{\varepsilon}^{(1)}f(\tau, \lambda):=&\, f(\tau+\varepsilon, \lambda)-f(\tau, \lambda),\\
   \delta_{\varepsilon}^{(2)}f(\tau, \lambda):=&\, f(\tau, \lambda+\varepsilon)-f(\tau, \lambda).
 \end{split}
 \end{equation} 
Define the remainder terms $R_{\varepsilon}^+(\tau, \lambda)$ and $R_{\varepsilon}^-(\tau, \lambda)$ as follows:  
 \begin{equation}\label{eq R}
\begin{split}
 R_{\varepsilon}^{\pm}(\tau, \lambda):=& \,\delta_{\pm\varepsilon}^{(1)}\delta_{\varepsilon}^{(2)}v(\tau, \lambda)-F(v(\tau, \lambda))\delta_{\pm\varepsilon}^{(1)}\delta_{\varepsilon}^{(2)}V(\tau, \lambda)\\
=&\, \left\{ v(\tau\pm\varepsilon, \lambda+\varepsilon)-v(\tau\pm\varepsilon, \lambda) - v(\tau, \lambda+\varepsilon)+v(\tau, \lambda)\right\}\\
& \,\, -F(v(\tau, \lambda))\left\{V(\tau\pm\varepsilon, \lambda+\varepsilon)-V(\tau\pm\varepsilon, \lambda)-V(\tau, \lambda+\varepsilon)+V(\tau, \lambda) \right\}.
\end{split}
\end{equation}
When $H=\frac12$,  it is established   in  Huang et al. \cite{HOO2024} that
 $$\left\| R_{\varepsilon}^{\pm}(\tau, \lambda)\right\|_{L^p(\Omega)}\le c(p)|\varepsilon|^{\frac{3}{2}},$$ 
 which    converges to $0$   as $\varepsilon\to 0$ at a rate   faster than 
 $$ \left\|Y(t, x + \varepsilon) -Y(t, x)\right\|_{L^p(\Omega)}\le c(p)  |\varepsilon|, $$
 where $c(p)\in(0,\infty)$.

The following  local linearization   result extends the previous work by  Huang et al. in \cite{HOO2024}. 
    \begin{theorem}\label{thm error 1} Assume that $H \in \left(\frac12, 1\right)$. For any $p \ge 1$, there exists a constant $c(p)>0$ such that 
\begin{align}\label{eq error}
\left\| R_{\varepsilon}^{\pm}(\tau, \lambda)\right\|_{L^p(\Omega)} \le c(p)\varepsilon^{2H+\frac{1}{2}},
\end{align}
holds uniformly for all $\tau > 0$, $\lambda \ge -\tau$, and for all sufficiently small $\varepsilon > 0$.
\end{theorem}

  As an immediate consequence of Theorem \ref{thm error 1}, we are thus led to the following bi-parameter local linearization result for the solution of Eq. \eqref{SWE} in the original coordinates. 
  
  \begin{theorem}\label{thm error 2} Assume that $H \in \left(\frac12, 1\right)$, and let $u$ and $U$ be the solutions to equations \eqref{SWE 2} and \eqref{SWE linear}, respectively. Then, for any $(t,x) \in \mathbb R_+ \times \mathbb R$ and any $p \ge 1$, it holds that
\begin{equation}\label{eq error 2}
\begin{split}
&\left\| \Delta_{\varepsilon}^{(1)} u(t, x) - F(u(t,x)) \Delta_{\varepsilon}^{(1)} U(t, x) \right\|_{L^p(\Omega)} \\
&\quad + \left\| \Delta_{\varepsilon}^{(2)} u(t, x) - F(u(t,x)) \Delta_{\varepsilon}^{(2)} U(t, x) \right\|_{L^p(\Omega)} \le c(p, t, x)   \varepsilon^{2H + \frac{1}{2}},
\end{split}
\end{equation}
uniformly for all sufficiently small $\varepsilon > 0$. Here, the difference operators $\Delta{\varepsilon}^{(1)}$ and $\Delta_{\varepsilon}^{(2)}$ are defined for a function $f(t, x)$ as follows:
\begin{align*}
\Delta_{\varepsilon}^{(1)}f(t,x)=&\, f(t,x+2\varepsilon)-f(t-\varepsilon, x+\varepsilon)-f(t+\varepsilon, x+\varepsilon)+f(t,x),\\
\Delta_{\varepsilon}^{(2)}f(t,x)=&\, f(t+2\varepsilon, x)-f(t+\varepsilon, x-\varepsilon)-f(t+\varepsilon, x+\varepsilon)+f(t,x).
\end{align*} 
\end{theorem}

       For every $N\ge1$ and $0\le i\le N$, let
   $\tau_i:=\frac{i}{N}, \lambda_j:=\frac{j}{N}$.   Define the rescaled quadratic  variation of   the process $v$ given in \eqref{SWE 2}  by 
  \begin{align}\label{eq V v}
Q_{N}(v):= N^{2H-1}\sum_{i=0}^{N-1}\sum_{j=0}^{N-1}\Delta_{i,j}(v)^2,
\end{align} 
with
\begin{align}\label{eq V ij}
\Delta_{i,j}(v)  := v\left(\tau_{i+1}, \lambda_{j+1}  \right)-v\left(\tau_{i+1}, \lambda_{j}\right)-v\left(\tau_{i}, \lambda_{j+1}  \right) +v\left(\tau_{i}, \lambda_{j}    \right).
\end{align}

 We have  the following asymptotic behavior of the sequence $\{Q_{N}(v)\}_{ N\ge 1}$, as $N\rightarrow\infty$.
  The idea is to approximate the increment $\Delta_{i,j}(v)$  by $F(v(\tau_i, \lambda_j))\Delta_{i,j}(V)$, where $V$ is given in \eqref{SWE 3}.   
    \begin{proposition}\label{prop variation converge}  
  Assume $H\in \left[\frac12, 1\right)$. There exists a constant $c>0$ such that  for every  fixed   $N\ge 2$, 
  \begin{align}\label{eq quad}
  \mathbb E\left[\left| Q_{N}(v)-  2^{H-\frac{5}{2}} \int_0^1\int_0^1 F^2\big(v(\tau, \lambda) \big) d\tau d\lambda \right|  \right]\le c N^{- H}.
  \end{align} 
   \end{proposition}
   
    The asymptotic limit of the quadratic variation $Q_N(v)$, characterized in Proposition \ref{prop variation converge}, paves the way for a direct estimation of the diffusion parameter. A comprehensive construction of the estimator   and its consistency is provided in Section \ref{QVPE}.

 We structure the remainder of this paper as follows. In Section 2, we present preliminary results on stochastic integration and 
properties of solutions to Equation \eqref{SWE}, adapted from \cite{D99, DQ}. In Section \ref{sec proof}, we   adapt techniques from \cite{HOO2024} to    prove Theorem \ref{thm error 1} concerning estimation error bounds. Finally,  we develop the quadratic variation analysis for the nonlinear stochastic wave equation and construct the corresponding parameter estimation framework in Section 4.

 \section{The noise and stochastic integral} 
 Following \cite{D99, DQ}, when $H\in (\frac12,1)$,  the covariance \eqref{Eq:noise_cov} is used to construct an inner product on the space $\mathcal D$ defined by 
\begin{align}
\begin{aligned}\label{Eq:noise_cov2}
 \langle \varphi, \psi\rangle := & \, \frac{1}{2\pi}\int_{\R_+} ds \int_\R \mu(d\xi)\, \mathscr{F}(\varphi(s, \cdot))(\xi) \overline{\mathscr{F}(\psi(s, \cdot))(\xi)}\\
  =  & \, \int_{\R_+} ds \int_\R dy \int_\R dy' \, \varphi(s, y) |y-y'|^{2H-2} \psi(s, y').
\end{aligned}
\end{align} 
 Let $\mathcal H$ be the completion of $\mathcal D$ with respect to the inner product $\langle \cdot, \cdot\rangle$. This space $\mathcal H$ will be the natural space of deterministic integrands with respect to $W$.  For any $g\in \mathcal H$, we say that $W(\varphi)$ is the Wiener integral of $\varphi$, denoted by 
$$
\int_0^{\infty}\int_{\mathbb R} g(t, x) W(dt, dx):= W(g).
$$
The space $\mathcal H$ contains all functions $g$ such that its Fourier transform in the space variable satisfies 
$$
\int_{0}^{\infty}   \int_\R     \left| \mathscr{F}(g(t, \cdot))(\xi) \right|^2    |\xi|^{1-2H} d\xi dt <\infty. 
$$
In particular, the space $\mathcal H$ contains all elements of the form ${\bf 1}_{[0,t]\times [0,x]}$, with $t>0$ and $x\in \mathbb R$. As a consequence of the representation   of the fractional Brownian  motion as a Wiener type integral with respect to a complex Brownian motion (see, instance, \cite[p. 257]{PT2000}), we have  
  \begin{align*}
& \E\left[W\left(\mathbf{1}_{[0,t]\times [0,x]}\right) W\left(\mathbf{1}_{[0,s]\times [0,y]}\right)\right] \\
= & \, \frac{1}{2} (t \wedge s) \left(|x|^{2H} + |y|^{2H} - |x-y|^{2H} \right).
\end{align*}
    This covariance corresponds to that of a standard Brownian motion in the time variable, while in the space variable it matches the covariance of a fractional Brownian motion with Hurst parameter  $H$.  The underlying filtration $\{\mathcal F_t, t\ge0\}$ generated by $W$ is 
\begin{align}\label{eq sigma filed}
\mathcal F_t=\sigma\left\{W({\bf 1}_{[0,s]} \varphi), s\in [0, t], \varphi\in  C_0^{\infty}\right\} \vee \mathcal N, 
\end{align}
 where $\mathcal N$ denotes the  class of  $\mathbb P$-null sets in $\mathcal F$.

   The solution to equation \eqref{SWE}  is understood in the mild sense. Specifically, for any $T>0$, we say that an adapted and 
 jointly measurable   process $u=\{u(t,x)\}_{t\ge0, x\in \mathbb R}$  is a solution to   \eqref{SWE} if, for all $(t,x)\in [0,T]\times \mathbb R$, the following holds:
\begin{equation}\label{eq solution1}
\begin{split}
u(t,x) =  \, \int_0^t\int_{\mathbb R} G(t-s, x-y) F(u(s,y))W(ds,dy),
   \end{split}
\end{equation}
where,  $G(t, x)$ is the   fundamental solution of the    wave equation in $\mathbb R$, given by 
\begin{align}\label{eq fund}
G(t,x)=\frac{1}{2}\mathbf 1_{\{|x|\le t\}}.
\end{align} 
As a result,   the last term in \eqref{eq solution1}  can be expressed as 
 \begin{equation}\label{eq solution}
\begin{split}
   \, \frac{1}{2}\iint_{\Delta(t,x)} F(u(s,y))W(ds,dy),
   \end{split}
\end{equation}
 where  
  \begin{align}
\Delta(t, x):=\left\{(s,y)\in \mathbb R_+\times \mathbb R: 0\le s\le t, \, |x-y|\le t-s \right\},
\end{align}
as illustrated  by the shaded area in Figure \ref{fig1}.
 
\begin{figure}
\begin{tikzpicture}
\path [fill=lightgray] (0,-2) -- (0,4) -- (3,1);
\draw [->] (0,-2.5) -- (0,4.6);
\draw [->] (0,0) -- (4,0);
\draw (0, -2) -- (0, 4) -- (3,1) -- (0,-2);
\draw [dashed] (3,0) node [below] {$t$} -- (3,1);
\draw [dashed] (0,1) node [left] {$x$} -- (3,1);
\draw [fill] (3,1) circle [radius=.05];
\draw [fill] (0,4) circle [radius=.05] node [left]{$x+t$};
\draw [fill] (0,-2) circle [radius=.05] node [left]{$x-t$};

\draw [fill] (0,4.6)   node [above]{$y$};
\draw [fill] (4.0,0)   node [right]{$s$};
\end{tikzpicture}
\caption{}\label{fig1}
\end{figure}

The existence, uniqueness, and regularity of the solution to \eqref{SWE} are guaranteed by the following results from \cite{D99, DS15}.   
\begin{proposition}\label{prop solution Holder}(\cite[Theorem 13]{D99},  \cite[Section 2.1]{DS15}) Assume that $H\in (\frac12, 1)$ and that $F$ is  Lipschitz continuous.  Then equation \eqref{SWE} admits  a unique mild solution $u=\{u(t,x)\}_{t\ge0, x\in \mathbb R}$.  Moreover, the following properties hold:
\begin{itemize}
 \item[(a)] For every  $T>0$ and  $1 \le p<\infty$,  
\begin{align}\label{eq moment bound}
\sup_{0\le t\le T}\sup_{x\in \mathbb R}\mathbb E\left[|u(t,x)|^p\right]<\infty. 
\end{align}
 
\item[(b)] For every  $T>0$  and $1 \le p<\infty$,  there exists a constant $c(p)>0$ such that 
\begin{align}\label{eq Holder}
\mathbb E\left[\left|u(t, x)-u(s, y)\right|^p \right]\le c(p)\left[|t-s|^{H}+|x-y|^{H} \right] ^{p},
\end{align}
for all $t, s\in [0,T]$ and $x, y\in \mathbb R$. 
\end{itemize}
  \end{proposition}

\section{Local linearization}\label{sec proof}
  
    The following   lemma  is taken from Lee and Xiao \cite{LX2022}.
   \begin{lemma}\label{Lem2}\cite[Lemmas 2.3 and 2.4]{LX2022}  Assume that $H\in \left(\frac12, 1\right)$.   Let $V$ be given by \eqref{SWE 3}.  
     For every $\tau>0, \lambda \ge-\tau$, and  $0 < \varepsilon< \tau$,  
\begin{equation}\label{Eq:rec-inc}
\begin{split}
 \mathbb E\left[\left(V(\tau\pm\varepsilon, \lambda + \varepsilon) - V(\tau\pm \varepsilon, \lambda) -V(\tau, \lambda + \varepsilon) + V(\tau, \lambda)\right)^2\right] 
  =  2^{H-\frac{5}{2}}\varepsilon^{2H+1}.
\end{split}
\end{equation}
\end{lemma}

\begin{figure}
\begin{tikzpicture}
\draw [->] (0,-3.75) -- (0,4) node [above] {$y$};
\draw [->] (0,0) -- (6.8,0) node [right] {$s$};
\draw [->] (0,0) -- (3.6,3.6) node [above right] {$\lambda$};
\draw [->] (0,0) -- (3.7,-3.7) node [right] {};

\draw [thick,domain=5.5/sqrt(2):8/sqrt(2)] plot (\x, {8/sqrt(2)-\x});
\draw [dashed,domain=4/sqrt(2):5.5/sqrt(2)] plot (\x, {8/sqrt(2)-\x});

\draw [thick,domain=3/sqrt(2):5.5/sqrt(2)] plot (\x, {3/sqrt(2)-\x});
\draw [dashed,domain=1.5/sqrt(2):4/sqrt(2)] plot (\x, {3/sqrt(2)-\x});

\draw [thick,domain=5.5/sqrt(2):8/sqrt(2)] plot (\x, {-8/sqrt(2)+\x});
\draw [dashed,domain=4/sqrt(2):5.5/sqrt(2)] plot (\x, {-8/sqrt(2)+\x});

\draw [thick,domain=3/sqrt(2):5.5/sqrt(2)] plot (\x, {-3/sqrt(2)+\x});
\draw [dashed,domain=1.5/sqrt(2):4/sqrt(2)] plot (\x, {-3/sqrt(2)+\x});

\draw [thick,domain=-2.5/sqrt(2):2.5/sqrt(2)] plot ({5.5/sqrt(2)}, \x); 

\draw [fill] ({4/sqrt(2)} ,{4/sqrt(2)}) circle [radius=.05] 

node [xshift=-0.6cm, yshift=0.25cm]{$\lambda+\varepsilon$};

\draw [fill] ({1.5/sqrt(2)},{1.5/sqrt(2)}) circle [radius=.05] 
node [xshift=-0.6cm, yshift=0.25cm]{$\lambda$};
\draw [fill] ({1.5/sqrt(2)},{-1.5/sqrt(2)}) circle [radius=.05] node [below left]{$\tau$};
\draw [fill] ({4/sqrt(2)},{-4/sqrt(2)}) circle [radius=.05] node [below left]{$\tau+\varepsilon$};
\node at (3.35,0.3) {$D_1$};
\node at (4.35,0.3) {$D_2$};

\draw [fill] ({3/sqrt(2)},{0}) circle [radius=.05] node[below=0.6mm]{$P_1$};
\draw [fill] ({5.5/sqrt(2)},{2.5/sqrt(2)}) circle [radius=.05] node [above=0.6mm]{$P_3$};
 
\draw [fill] ({5.5/sqrt(2)},{-2.5/sqrt(2)}) circle [radius=.05] node [below=0.6mm]{$P_2$};

\draw [fill] ({8/sqrt(2)},{0}) circle [radius=.05] node [below=0.6mm]{$P_4$};
   
\end{tikzpicture}
\caption{}\label{fig2}
\end{figure}
 
 Here are two elementary facts that will be used in the proof of Theorem \ref{thm error 1}. 
   \begin{lemma}
 For any $a<b$ and $H\in \left(\frac12, 1\right)$, we have 
 \begin{itemize}
\item[(a)] \begin{equation}\label{eq ab1}
 \int_a^bdx \int_a^bdy |x-y|^{2H-2}    = \frac{1}{(2H-1)H} (b-a)^{2H}.
 \end{equation}
 \item[(b)]
  \begin{equation}\label{eq ab2}
\int_{a}^{\frac{a+b}{2}}dx\int_{\frac{a+b}{2}}^bdy  |x-y|^{2H-2} =\frac{ 2^{2H-1} - 1}{H(2H-1)}  \left(\frac{b-a}{2}\right)^{2H}.
 \end{equation}
 \end{itemize}
 \end{lemma} 
  \begin{proof} The identity  \eqref{eq ab1} follows by  direct computation. We now prove \eqref{eq ab2}.  Using the change of variables $$x= \left(\frac{b-a}{2}\right)u+a,\,  y=\left(\frac{b-a}{2}\right)v +\frac{a+b}{2},$$ 
  we obtain
 \begin{align*}
  \int_{a}^{\frac{a+b}{2}}dx\int_{\frac{a+b}{2}}^bdy  |x-y|^{2H-2} 
  = \, \left(\frac{b-a}{2}\right)^{2H}\int_{0}^{1}dv\int_{0}^1du (1+v-u)^{2H-2}.
 \end{align*}
 Note that 
 \begin{align*}
\int_{0}^{1}  dv  \int_{0}^{1} du (1 + v - u)^{2H-2}    =&\,   \frac{1}{2H-1} \int_{0}^{1}    \left((1+v)^{2H-1} - v^{2H-1}\right)dv  \\
=&\,  \frac{1}{2H-1} \left( \frac{2^{2H} - 1}{2H} - \frac{1}{2H}\right)\\
=&\, \frac{ 2^{2H-1} - 1}{H(2H-1)}.
 \end{align*}
 
 The proof is complete.   
  \end{proof}

\begin{proof}[Proof of Theorem \ref{thm error 1}]    
The proof follows a strategy similar to that of   Theorem 1.3 in \cite{HOO2024}.   By symmetry, we only need to prove \eqref{eq error} for $R_{\varepsilon}(\tau, \lambda)$.   By  \eqref{SWE 2}  and Proposition \ref{prop solution Holder},  it follows  that for any   $p\ge1$, there exists a constant $c(p)>0$ such that
\begin{align}\label{eq Holder v}
\|v(\tau_1, \lambda_1)-v(\tau_2, \lambda_2)\|_{L^p(\Omega)}\le c(p)\left[|\tau_1-\tau_2|^{H}+|\lambda_1-\lambda_2|^{H}\right],
\end{align}
 for all $\tau_1, \tau_2>0, \lambda_1\ge -\tau_1, \lambda_2\ge -\tau_2$.
  To ensure that   the integrand is  adapted, we decompose the domain of integration into two triangular regions. A     similar decomposition technique  to recover adaptedness is employed  in \cite[pp. 21]{CHKK19}.  For every fixed $\tau>0, \lambda\ge -\tau$, and $\varepsilon>0$, define  the regions $D_1$ and $D_2$ as follows: 
\begin{align*}
D_1:= & \, \Bigg\{(s, y)\in \mathbb R_+\times \mathbb R:\,  \frac{\tau+\lambda}{\sqrt 2} <  \,  s\le \frac{\tau+\lambda+\varepsilon}{\sqrt 2},\\
 & \,\,\,\,\,\,\,\,\,\,\,    \,\,\,\,\,\,\,\,\,\,\,  -s+\sqrt{2}\lambda <  y < s-\sqrt{2}\tau \Bigg\},\\
D_2:=&  \, \Bigg\{(s,y)\in \mathbb R_+\times \mathbb R :\,  \frac{\tau+\lambda+\varepsilon}{\sqrt 2} <  \, s \le \frac{\tau+\lambda+2\varepsilon}{\sqrt 2},\\
 & \,\,\,\,\,\,\,\,\,\,\,    \,\,\,\,\,\,\,\,\,\,\,    s-\sqrt{2}(\tau+\varepsilon)\le \, y\le -s+\sqrt 2(\lambda+ \varepsilon) \Bigg\}.
\end{align*}
 The  sets $D_1$ and $D_2$ are  illustrated in Figure \ref{fig2}. Specifically,   $D_1$ is the triangular region with vertices:
\begin{equation}\label{eq P}
\begin{split}  
   P_1 :=&\, \left(\frac{\tau+\lambda}{\sqrt 2}, \frac{-\tau+\lambda}{\sqrt 2}\right),\\
     P_2:=&\,  \left(\frac{\tau+\lambda+\varepsilon}{\sqrt 2}, \frac{-\tau+\lambda-\varepsilon}{\sqrt 2}\right),\\
P_3:=&\, \left(\frac{\tau+\lambda+\varepsilon}{\sqrt 2}, \frac{-\tau+\lambda+\varepsilon}{\sqrt 2}\right),\\
  \end{split}  
\end{equation}
and $D_2$ is the triangular region with vertices:
\begin{align*}
 P_2 , \, P_3,  \text{ and }\  P_4 :=&\,  \left(\frac{\tau+\lambda+2\varepsilon}{\sqrt 2}, \frac{-\tau+\lambda}{\sqrt 2}\right).
\end{align*}

Using equations  \eqref{eq transform} and \eqref{eq solution},  we decompose  the integral as follows:
\begin{align*}
 &  \frac{1}{2}\left(\iint_{D_1}+\iint_{D_2} \right)\left\{F(u(s,y))-F(u(P_1)) \right\}W(ds,dy)\\
=&\,  \frac{1}{2} \iint_{D_1} \left\{F(u(s, y))-F(u(P_1))\right\}W(ds,dy)\\
&\, +\frac{1}{2}   \left\{F(u(P_2))- F(u(P_1)) \right\} \iint_{D_2} W(ds,dy)\\
&\, +   \frac{1}{2} \iint_{D_2} \left\{F(u(s,y))-F(u(P_2)) \right\}W(ds,dy)\\
=:&\,  I_1+I_2+I_3.
\end{align*} 

  By the Burkholder–Davis–Gundy (BDG, for short) inequality,       \eqref{Eq:noise_cov},   \eqref{eq Holder},  and the Lipschitz continuity of $F$, we have 
\begin{align*}
 \mathbb{E}\left[\left|I_1\right|^p\right]^{\frac{2}{p}}
 \lesssim &\,   \int_{\frac{\tau+\lambda}{\sqrt 2} }^ \frac{\tau+\lambda+\varepsilon}{\sqrt 2}ds \int_{-s+ \sqrt{2}\lambda}^{s -\sqrt{2}\tau } dy \int_{-s + \sqrt{2}\lambda}^{s -\sqrt{2}\tau}  d\tilde{y}  |y - \tilde{y}|^{2H-2}\\
& \ \ \ \cdot  \left\| F(u(s,y)) - F(u(P_1)) \right\|_{L^p(\Omega)}   
   \left\| F(u(s,\tilde y)) -  F(u(P_1)) \right\|_{L^p(\Omega)}  \\
\lesssim &\,   \int_{\frac{\tau+\lambda}{\sqrt 2} }^ \frac{\tau+\lambda+\varepsilon}{\sqrt 2}ds\int_{-s + \sqrt{2}\lambda}^{s -\sqrt{2}\tau } dy \int_{-s + \sqrt{2}\lambda}^{s -\sqrt{2}\tau}  d\tilde{y}  |y - \tilde{y}|^{2H-2}\\
&\,   \cdot \left( \left| s - \frac{\tau+\lambda}{\sqrt{2}} \right| + \left| y - \frac{-\tau+\lambda}{\sqrt{2}} \right| \right)^{H}  \cdot \left( \left| s- \frac{\tau+\lambda}{\sqrt{2}} \right| + \left| \tilde y - \frac{-\tau+\lambda}{\sqrt{2}} \right| \right)^{H}.
  \end{align*}
  Here and below, $a\lesssim b$ means that there exists a constant $c>0$ such that $a\le c b$.
  
We further  divide the domain  of  integration  into  four parts:
\begin{align*}
    \left(  \int_{-s + \sqrt{2}\lambda}^{\frac{\lambda-\tau}{\sqrt 2} }  +   \int_{\frac{\lambda-\tau}{\sqrt 2}}^{s-\sqrt{2}\tau }    \right) dy \cdot
  \left(  \int_{-s + \sqrt{2}\lambda}^{\frac{\lambda-\tau}{\sqrt 2} }  +   \int_{\frac{\lambda-\tau}{\sqrt 2}}^{s -\sqrt{2}\tau }    \right)   d\tilde{y}.  
  \end{align*}
  
For the first term, we have
\begin{align*}
 & \int_{\frac{\tau+\lambda}{\sqrt 2} }^ \frac{\tau+\lambda+\varepsilon}{\sqrt 2}ds   \int_{-s + \sqrt{2}\lambda}^{\frac{\lambda-\tau}{\sqrt 2} } dy \int_{-s + \sqrt{2}\lambda}^{\frac{\lambda-\tau}{\sqrt 2} } d\tilde{y} 
  |y - \tilde{y}|^{2H-2} \\
  &\, \cdot \left( \left|s - \frac{\tau+\lambda}{\sqrt{2}} \right| + \left| y - \frac{-\tau+\lambda}{\sqrt{2}} \right| \right)^{H} 
  \cdot \left( \left| s - \frac{\tau+\lambda}{\sqrt{2}} \right| + \left| \tilde y - \frac{-\tau+\lambda}{\sqrt{2}} \right| \right)^{H}\\
   =&\,  \int_{\frac{\tau+\lambda}{\sqrt 2} }^ \frac{\tau+\lambda+\varepsilon}{\sqrt 2}ds   \int_{-s+ \sqrt{2}\lambda}^{\frac{\lambda-\tau}{\sqrt 2} } dy \int_{-s + \sqrt{2}\lambda}^{\frac{\lambda-\tau}{\sqrt 2} } d\tilde{y} 
  |y - \tilde{y}|^{2H-2}  \cdot \left(s- y-\sqrt{2}\tau  \right)^{H}  \cdot \left(  s- \tilde y-\sqrt{2}\tau  \right)^{H} \\
     \lesssim  &\,    \int_{\frac{\tau+\lambda}{\sqrt 2} }^ \frac{\tau+\lambda+\varepsilon}{\sqrt 2}ds  \left (2s-\sqrt{2}(\tau+\lambda)\right)^{2H}     \int_{-s + \sqrt{2}\lambda}^{\frac{\lambda-\tau}{\sqrt 2} } dy \int_{-s + \sqrt{2}\lambda}^{\frac{\lambda-\tau}{\sqrt 2} } d\tilde{y}
  |y - \tilde{y}|^{2H-2}  \\
  \lesssim &\,   \int_{\frac{\tau+\lambda}{\sqrt 2} }^ \frac{\tau+\lambda+\varepsilon}{\sqrt 2}ds  \left ( s- \frac{\tau+\lambda}{\sqrt{2}}\right)^{4H}  \\
=&\, \frac{1}{(1+4H) 2^{\frac12+2H}}   \varepsilon^{1+4H},
 \end{align*}
where     \eqref{eq ab1} is used  in the   last  second step.

For the second term,   changing   variables yields that 
\begin{align*}
 &   \int_{\frac{\tau+\lambda}{\sqrt 2} }^ \frac{\tau+\lambda+\varepsilon}{\sqrt 2}ds      \int_{-s + \sqrt{2}\lambda}^{\frac{\lambda-\tau}{\sqrt 2} } dy   \int_{\frac{\lambda-\tau}{\sqrt 2}}^{s -\sqrt{2}\tau }   d\tilde{y} 
  |y - \tilde{y}|^{2H-2} \\
    &\, \cdot \left( \left|s - \frac{\tau+\lambda}{\sqrt{2}} \right| + \left| y - \frac{-\tau+\lambda}{\sqrt{2}} \right| \right)^{H} 
  \cdot \left( \left| s - \frac{\tau+\lambda}{\sqrt{2}} \right| + \left| \tilde y- \frac{-\tau+\lambda}{\sqrt{2}} \right| \right)^{H}\\
  =&\, \int_{\frac{\tau+\lambda}{\sqrt 2} }^ \frac{\tau+\lambda+\varepsilon}{\sqrt 2}ds     \int_{-s+ \sqrt{2}\lambda}^{\frac{\lambda-\tau}{\sqrt 2} } dy      \int_{\frac{\lambda-\tau}{\sqrt 2}}^{s -\sqrt{2}\tau }   d\tilde{y} 
  |y - \tilde{y}|^{2H-2}   \cdot \left(  s - y-\sqrt{2}\tau  \right)^{H}     \cdot \left(  s+\tilde y-\sqrt{2} \lambda \right)^{H}\\
\le &  \int_{\frac{\tau+\lambda}{\sqrt 2} }^ \frac{\tau+\lambda+\varepsilon}{\sqrt 2}ds      \left (2s-\sqrt{2}(\tau+\lambda)\right)^{2H}   \int_{-s + \sqrt{2}\lambda}^{\frac{\lambda-\tau}{\sqrt 2} } dy      \int_{\frac{\lambda-\tau}{\sqrt 2}}^{s -\sqrt{2}\tau }   d\tilde{y}   |y - \tilde{y}|^{2H-2}   \\
  \lesssim &\,   \int_{\frac{\tau+\lambda}{\sqrt 2} }^ \frac{\tau+\lambda+\varepsilon}{\sqrt 2}ds  \left (s- \frac{\tau+\lambda}{\sqrt{2}}\right)^{4H}  \\
=&\,  \frac{1}{(1+4H) 2^{\frac12+2H}}  \varepsilon^{1+4H},
 \end{align*} 
 where    \eqref{eq ab2} is used  in the   last  second step.  

 By symmetry, the third term equals to the second.   
  For the fourth term, we have
\begin{align*}
& \int_{\frac{\tau+\lambda}{\sqrt 2} }^ \frac{\tau+\lambda+\varepsilon}{\sqrt 2}ds    \int_{\frac{\lambda-\tau}{\sqrt 2}}^{s -\sqrt{2}\tau }    dy    \int_{\frac{\lambda-\tau}{\sqrt 2}}^{s-\sqrt{2}\tau }     d\tilde{y} |y - \tilde{y}|^{2H-2} \\
  &\, \cdot \left( \left| s - \frac{\tau+\lambda}{\sqrt{2}} \right| + \left| y - \frac{-\tau+\lambda}{\sqrt{2}} \right| \right)^{H} 
  \cdot \left( \left| s - \frac{\tau+\lambda}{\sqrt{2}} \right| + \left| \tilde y - \frac{-\tau+\lambda}{\sqrt{2}} \right| \right)^{H}\\
  =&\, \int_{\frac{\tau+\lambda}{\sqrt 2} }^ \frac{\tau+\lambda+\varepsilon}{\sqrt 2}ds   \int_{\frac{\lambda-\tau}{\sqrt 2}}^{s -\sqrt{2}\tau }    dy   \int_{\frac{\lambda-\tau}{\sqrt 2}}^{s -\sqrt{2}\tau }     d\tilde{y} |y - \tilde{y}|^{2H-2}  \cdot 
  \left(   s +y-\sqrt{2} \lambda \right)^{H}  \left(   s+\tilde y-\sqrt{2}  \lambda \right)^{H}  \\
\lesssim &  \int_{\frac{\tau+\lambda}{\sqrt 2} }^ \frac{\tau+\lambda+\varepsilon}{\sqrt 2}ds       \left (2s-\sqrt{2}(\tau+\lambda)\right)^{2H} \int_{\frac{\lambda-\tau}{\sqrt 2}}^{s -\sqrt{2}\tau }    ds    \int_{\frac{\lambda-\tau}{\sqrt 2}}^{s -\sqrt{2}\tau }     d\tilde{y} |y - \tilde{y}|^{2H-2} \\
   = &\,    \int_{\frac{\tau+\lambda}{\sqrt 2} }^ \frac{\tau+\lambda+\varepsilon}{\sqrt 2}ds \left ( s- \frac{\tau+\lambda}{\sqrt{2}}\right)^{4H}  \\
=&\, \frac{1}{(1+4H) 2^{\frac12+2H}} \varepsilon^{1+4H},
 \end{align*}
where    \eqref{eq ab1} is used  in  last  second step.

Using the  H\"older  inequality,  we have 
\begin{equation}
\begin{split}
 \mathbb E\left[|I_2|^p\right]^{\frac{2}{p}} \lesssim &\,  \mathbb E\left[ \left| u(P_2)- u(P_1)\right|^{2p}\right]^{\frac{1}{p}}  \cdot \mathbb E\left[ \left|\iint_{D_2} W(ds,dy)\right|^{2p}\right]^{\frac1p}.
\end{split}
\end{equation} 
 By the Lipschitz continuity of $F$,  \eqref{eq Holder}, and  \eqref{eq P}, we have 
 \begin{align*}
 \mathbb E\left[ \left| u(P_2)- u(P_1)\right|^{2p}\right]^{\frac{1}{p}}  \lesssim \varepsilon^{2H}.
 \end{align*}
By   \eqref{eq ab2}, we have 
\begin{align*}
\mathbb E\left[ \left|\iint_{D_2} W(ds,dy)\right|^{2}\right]=&\,  \int_{\frac{\tau+\lambda+\varepsilon}{\sqrt 2}}^{   \frac{\tau+\lambda+2\varepsilon}{\sqrt 2}}ds\int_{
  s-\sqrt{2}(\tau+\varepsilon)}^{  -s+\sqrt 2(\lambda+ \varepsilon)}dy  \int_{   s-\sqrt{2}(\tau+\varepsilon)}^{  -s+\sqrt 2(\lambda+ \varepsilon)}d\tilde y   |y- \tilde{y}|^{2H-2}\\
   \lesssim &\,  \varepsilon^{1+2H}.
 \end{align*}

 Finally, for $I_3$, applying the BDG  inequality, the Lipschitz continuity of $F$, \eqref{Eq:noise_cov}, \eqref{eq Holder},  and \eqref{eq ab2}, we have 
\begin{align*}
 \mathbb{E}\left[\left|I_3\right|^p\right]^{\frac{2}{p}}
 \lesssim &\,   \int_{\frac{\tau+\lambda+\varepsilon}{\sqrt 2}}^{\frac{\tau+\lambda+2\varepsilon}{\sqrt 2}}ds \int_{s-\sqrt{2}(\tau+\varepsilon)}^{-s+\sqrt 2(\lambda+ \varepsilon)} dy \int_{s-\sqrt{2}(\tau+\varepsilon)}^{-s+\sqrt 2(\lambda+ \varepsilon) }  d\tilde{y}  |y - \tilde{y}|^{2H-2}\\
& \ \ \ \cdot  \left\| F(u(s,y)) - F(u(P_2)) \right\|_{L^p(\Omega)}   
   \left\| F(u(s,\tilde y)) -  F(u(P_2)) \right\|_{L^p(\Omega)}  \\
\lesssim &\,   \int_{\frac{\tau+\lambda+\varepsilon}{\sqrt 2}}^{\frac{\tau+\lambda+2\varepsilon}{\sqrt 2}}ds \int_{s-\sqrt{2}(\tau+\varepsilon)}^{-s+\sqrt 2(\lambda+ \varepsilon)} dy \int_{s-\sqrt{2}(\tau+\varepsilon)}^{-s+\sqrt 2(\lambda+ \varepsilon) }  d\tilde{y}  |y - \tilde{y}|^{2H-2}\\
&\,   \cdot \left(  s+y-\sqrt{2}\lambda \right)^{H}  \cdot \left(  s+\tilde y-\sqrt{2}\lambda  \right)^{H}\\
 \lesssim &\,   \varepsilon^2 \int_{\frac{\tau+\lambda+\varepsilon}{\sqrt 2}}^{\frac{\tau+\lambda+2\varepsilon}{\sqrt 2}}ds \int_{s-\sqrt{2}(\tau+\varepsilon)}^{-s+\sqrt 2(\lambda+ \varepsilon)} dy \int_{s-\sqrt{2}(\tau+\varepsilon)}^{-s+\sqrt 2(\lambda+ \varepsilon) }  d\tilde{y}  |y - \tilde{y}|^{2H-2}\\
 \lesssim&\,       \varepsilon^2 \int_{\frac{\tau+\lambda+\varepsilon}{\sqrt 2}}^{\frac{\tau+\lambda+2\varepsilon}{\sqrt 2}}ds  \left(\frac{\tau+\lambda+2\varepsilon}{\sqrt 2}-s\right)^{2H}\\
 \lesssim&\,    \varepsilon^{1+4H}.
  \end{align*}

The proof is complete. 
\end{proof}

\section{Quadratic variations and parameter estimation}\label{QVPE}

  The development of statistical inference for stochastic partial differential equations   is  largely motivated by the need to calibrate mathematical models using empirical data.    In many physical systems, including  wave propagation in heterogeneous media, the stochastic wave equation serves as   a fundamental framework. Accurate  estimation  of key parameters, especially the   diffusion constant or noise intensity, is essential for    enhancing the  predictive power of these models.    Although parameter    estimation under  additive noise has been extensively studied (see, e.g., \cite{AGT2022, KT2018, KTZ2018, Shev2023};  we also refer to      the monograph \cite{T2023} for   a comprehensive overview),  the multiplicative noise case remains less explored despite its greater practical relevance and theoretical challenges.

  \subsection{Parameter estimates}
In this section, we address the problem of estimating the diffusion parameter in the following stochastic wave equation:
    \begin{equation}\label{SWE para}
\begin{cases}
\vspace{6pt}
\displaystyle{\frac{\partial^2}{\partial t^2} u_{\theta }(t, x) =\frac{\partial^2}{\partial x^2} u_{\theta }(t, x) + \theta  F(u_{\theta}(t,x))\dot W(t, x), \quad t\ge 0, \, x \in \mathbb{R},}\\
\displaystyle{u_{\theta }(0, x) = 0, \quad \frac{\partial}{\partial t} u_{\theta}(0, x) = 0},
\end{cases}
\end{equation}
where $\dot W$ is a     space-time noise that is white in time and colored in space, and $F$ is a Lipschitz continuous function, as stated in \eqref{SWE}.

We investigate the estimation of the parameter $\theta$ from discrete observations of the solution $u_{\theta}$ over a space-time grid.
 Our analysis provides a partial resolution to the open problem posed in \cite[Remark 5.2]{AGT2022}. Unlike  the approach in \cite[Section 4]{TZ2025}, which relies on  discrete temporal  observations at a fixed spatial location,  the proposed methodology   fully exploits the   spatiotemporal structure of the data.

For every $\tau>0$ and $\lambda\ge -\tau$,    let
       \begin{align}\label{eq v theta}
   v_{\theta}(\tau, \lambda):=  u_{\theta}\left(\frac{\tau+\lambda}{\sqrt{2}}, \frac{-\tau+\lambda}{\sqrt{2}}\right). 
\end{align}
     For every $N\ge1$, let 
       \begin{align}\label{eq V v}
Q_{N}(v_{\theta}):= N^{2H-1}\sum_{i=0}^{N-1}\sum_{j=0}^{N-1}\Delta_{i,j}(v_{\theta})^2,
\end{align}
where 
$$
\Delta_{i,j}(v_{\theta})  := v_{\theta}\left(\tau_{i+1}, \lambda_{j+1}  \right)-v_{\theta}\left(\tau_{i+1}, \lambda_{j}\right)-v_{\theta}\left(\tau_{i}, \lambda_{j+1}  \right) +v_{\theta}\left(\tau_{i}, \lambda_{j}    \right), 
$$ 
 and    $\tau_i:=\frac{i}{N}, \lambda_j:=\frac{j}{N}$. 
 
  We postpone the proof of Proposition~\ref{prop variation converge}  to Subsection \ref{Sec QV}. Using Proposition~\ref{prop variation converge},  we know that
  \begin{align}\label{eq variation converge theta}
  Q_{N}(v_{\theta})\longrightarrow  \theta^2 2^{H-\frac{5}{2}}  \int_0^1\int_0^1 F^2\big(v_{\theta}(\tau, \lambda) \big) d\tau d\lambda,
  \end{align}
 in  $L^{1}(\Omega)$, as $N\rightarrow\infty$.  
This convergence allows us to  construct a consistent  estimator of the noise parameter $\theta$ based on the observations
$$\left\{u_{\theta}\left(\frac{\frac{i}{N}+\frac{j}{N} }{\sqrt 2}, \frac{-\frac{i}{N}+\frac{j}{N} }{\sqrt 2} \right);\, i, j = 0,1, \cdots, N\right\}.$$
  
Given the limit in \eqref{eq variation converge theta},  a natural approach  to estimating  $\theta$ is to use  the  approximation 
    \begin{align}\label{eq appr1}
   \theta  \approx \sqrt{\frac{Q_{N}(v_{\theta}) }{2^{H-\frac{5}{2}} \int_0^1\int_0^1 F^2\big(v_{\theta}(\tau, \lambda) \big) d\tau d\lambda}}.
    \end{align}
    To evaluate the  denominator,   we discretize the  double integral via  a   Riemann  sum:
    $$
    \int_0^1\int_0^1 F^2\big(v(\tau, \lambda) \big) d\tau d\lambda\approx N^{-2} \sum_{i=0}^{N-1} \sum_{j=0}^{N-1} F^2\big(v(\tau_i, \lambda_j) \big).
    $$
    This leads to the following estimator for    the diffusion parameter $\theta$:  for every $N\ge1$,  
      \begin{align}\label{eq appr2}
    \hat \theta_N := \sqrt{\frac{Q_{N}(v_{\theta})}{ 2^{H-\frac{5}{2}} N^{-2} \sum_{i=0}^{N-1} \sum_{j=0}^{N-1} F^2\big(v_{\theta}(\tau_i, \lambda_j) \big)}}.
    \end{align}
    
   From Proposition  \ref{prop variation converge}, we directly obtain the   following  consistency result. 
  \begin{corollary}\label{coro est}  Assume that $H\in \left[\frac12, 1\right)$. 
 The estimator $  \hat \theta_N$ defined in  \eqref{eq appr2} is weakly consistent, i.e., 
$$ \hat \theta_N \stackrel{\mathbb P}{\longrightarrow}   \theta,  \ \ \ \ \text{as } N\rightarrow\infty.
 $$
\end{corollary}
      \begin{proof}
By Proposition  \ref{prop variation converge},  $Q_{N}(v_{\theta})$ converges in probability to  $ \theta^2 \int_0^1\int_0^1 F^2\big(v_{\theta}(\tau, \lambda)\big)d\tau d\lambda$. It remains to show that 
$$
I:=    \int_0^1\int_0^1 F^2\big(v_{\theta}(\tau, \lambda)\big)d\tau d\lambda  -\frac{1}{N^2} \sum_{i=0}^{N-1} \sum_{j=0}^{N-1} F^2\big(v_{\theta}(\tau_i, \lambda_j) \big) \stackrel{\mathbb P}{\longrightarrow} 0.
 $$
Note that 
   \begin{align*}
I= & \,   \sum_{i=0}^{N-1} \sum_{j=0}^{N-1}   \int_{\tau_i}^{\tau_{i+1}}\int_{\tau_j}^{\tau_{j+1}} \left[ F^2\big(v_{\theta}(\tau, \lambda)\big)-F^2\big(v_{\theta}(\tau_i, \lambda_j) \big)\right]d\tau d\lambda.
 \end{align*}
 For each $\tau, \lambda\in  [\tau_i, \tau_{i+1}) \times [\tau_j, \tau_{j+1})$, the   Lipschitz continuity of $F$ and the H\"older continuity of $v_{\theta}$ imply   
  \begin{align*}
& \mathbb E\left[    \left|F^2\big(v_{\theta}(\tau, \lambda)\big)-F^2\big(v_{\theta}(\tau_i, \lambda_j) \big)\right|  \right]\\
\le &\,       \left\{ \mathbb E\left[ \left|F\big(v_{\theta}(\tau, \lambda)\big)+F\big(v_{\theta}(\tau_i, \lambda_j) \big)\right|^2 \right]\right\}^{1/2}  \left\{ \mathbb E\left[ \left|F\big(v_{\theta}(\tau, \lambda)\big)-F\big(v_{\theta}(\tau_i, \lambda_j) \big)\right|^2\right]\right\}^{1/2}     \\
   \lesssim &\,   N^{-H}. 
 \end{align*}
Therefore,  we have
$$
\mathbb E\left[I \right] \lesssim N^{-H},
$$
which implies convergence in probability.  
\end{proof}

  \subsection{Quadratic  variations} \label{Sec QV}
  For every $N\ge1$ and $0\le i\le N$, define the grid points
   $\tau_i:=\frac{i}{N}, \lambda_j:=\frac{j}{N}$,   and the second-order increment
  $$
\Delta_{i,j}(V) := V\left(\tau_{i+1}, \lambda_{j+1}  \right)-V\left(\tau_{i+1}, \lambda_{j}\right)-V\left(\tau_{i}, \lambda_{j+1}  \right) +V\left(\tau_{i}, \lambda_{j}    \right).
$$
  The corresponding quadratic variation statistic is defined as 
   \begin{align}\label{eq V V}
Q_{N}(V):= N^{2H-1}\sum_{i=0}^{N-1}\sum_{j=0}^{N-1}\Delta_{i,j}(V)^2.
\end{align}

  We first  study  the asymptotic behavior of the sequence   $\{Q_{N}(V)\}_{ N\ge 1}$  as $N\rightarrow\infty$.
  	 	   
	\begin{lemma}\label{lem linear app} Assume that $H\in \left[\frac12, 1\right)$. Then, there exists a constant $c>0$ satisfying that    
		\begin{equation}\label{2f-10}
			\mathbb{E}\left[ \left| Q_{N}(V) -2^{H-\frac{5}{2}}\right|^2   \right]  \lesssim N^{-1}.
					\end{equation} 
	\end{lemma}
 \begin{proof}
From  equation \eqref{Eq:rec-inc},  the random variable  $N^{H-\frac12} \Delta_{i,j}(V)$  follows a  Gaussian distribution with mean zero and   variance $2^{H-\frac{5}{2}}N^{-2}$.   Using  standard moment bounds for  Gaussian random variables, we obtain that for all $i, j\geq 1$ and $p\in \mathbb N$,	 	\begin{equation}\label{2f-2}
			\mathbb{E}\left[ \left(N^{2H-1} \left( \Delta_{i,j}(V) \right)^{2} -2^{H-\frac{5}{2}}N^{-2}\right) ^{2p}\right]\lesssim  N^{-4p}.
		\end{equation} 
	
	Due to the temporal independence  of the  noise $\{W(t, x)\}_{t\ge0, x\in\mathbb R}$, which is white in time and colored in space,   	  the increments 
  $\Delta_{i,j}(V) $ and $\Delta_{i',j'}(V) $ are independent whenever  $|i-i'|\ge2$.  Applying  the  Cauchy-Schwarz  inequality and the moment estimate  \eqref{2f-2},  we derive
       \begin{align*}
&   \mathbb{E}\left[   \left(\sum_{i=0}^{N-1}\sum_{j=0}^{N-1}  \left( N^{2H-1}  \left(\Delta_{i,j}(V) \right)^{2} -2^{H-\frac{5}{2}}N^{-2}\right)  \right)^{2}\right]   \\
  \le &\, 4  \sum_{i=0}^{N-1}   \mathbb{E} \left[   \left(  \sum_{j=0}^{N-1}  \left( N^{2H-1}  \left(\Delta_{i,j}(V) \right)^{2} -2^{H-\frac{5}{2}}N^{-2}\right)  \right)^{2}\right] \\
  \le & \, 4 N \sum_{i=0}^{N-1}   \sum_{j=0}^{N-1}   \mathbb{E} \left[  \left( N^{2H-1}  \left(\Delta_{i,j}(V) \right)^{2} -2^{H-\frac{5}{2}}N^{-2}  \right)^{2}\right]\\
\lesssim  &\,  N^{-1}.
    \end{align*} 
   
    The proof is complete. 
     \end{proof}

   Next, we prove Proposition \ref{prop variation converge} by applying the  approximation of the increment $\Delta_{i,j}(v)$  by $F(v(\tau_i, \lambda_j))\Delta_{i,j}(V)$, as established  in Theorem \ref{thm error 1}.
	 	      
   \begin{proof} [Proof of Proposition \ref{prop variation converge}]
We decompose the target difference into three parts:
\begin{equation}
\begin{split}
&Q_{N}(v)-     \int_0^1\int_0^1 F^2\big(v(\tau, \lambda) \big) d\tau d\lambda\\
 =&\, N^{2H-1}\sum_{i=0}^{N-1}\sum_{j=0}^{N-1}\left\{\Delta_{i,j}(v)^2 - F^2(v(\tau_i, \lambda_j)) \Delta_{i,j}(V)^2 \right\}\\
 & +\sum_{i=0}^{N-1}\sum_{j=0}^{N-1} F^2(v(\tau_i, \lambda_j)) \left\{N^{2H-1} \Delta_{i,j}(V)^2   -2^{H-\frac{5}{2}}N^{-2} \right\}\\
  &+2^{H-\frac{5}{2}} \left[\sum_{i=0}^{N-1}\sum_{j=0}^{N-1}  F^2(v(\tau_i, \lambda_j))  N^{-2} -  \int_0^1\int_0^1 F^2\big(v(\tau, \lambda) \big) d\tau d\lambda\right]\\
  =:&\, I_1+I_2+I_3.
 \end{split}
 \end{equation}
 
\noindent\textbf{Step 1:  Estimation of  $I_1$.}
  By  the Cauchy-Schwarz inequality,      
 \begin{align*}
&\mathbb E \left[\left|\Delta_{i,j}(v)^2 - F^2(v(\tau_i, \lambda_j)) \Delta_{i,j}(V)^2\right| \right]\\
\le &\,   \left( \mathbb E \left[\left|\Delta_{i,j}(v)-   F (v(\tau_i, \lambda_j)) \Delta_{i,j}(V)\right|^2 \right]\right)^{\frac12} \left( \mathbb E \left[\left|\Delta_{i,j}(v)+   F (v(\tau_i, \lambda_j)) \Delta_{i,j}(V)\right|^2 \right]\right)^{\frac12} .
 \end{align*}
 From Theorem \ref{thm error 1}, we obtain the bound for the first factor:
  \begin{align*}
  \left( \mathbb E \left[\left|\Delta_{i,j}(v)-   F (v(\tau_i, \lambda_j)) \Delta_{i,j}(V)\right|^2 \right]\right)^{\frac12} 
 \lesssim  \,   N^{-\frac 12-2H}.
 \end{align*}
  For the second factor, applying      Minkowski's inequality,  H\"older's inequality, \eqref{Eq:rec-inc},  and  Theorem \ref{thm error 1} again, we have
  \begin{align*}
&  \left( \mathbb E \left[\left|\Delta_{i,j}(v)+   F (v(\tau_i, \lambda_j)) \Delta_{i,j}(V)\right|^2 \right]\right)^{\frac12}\\
  \le &\, \left( \mathbb E \left[\left|\Delta_{i,j}(v) \right|^2 \right]\right)^{\frac12}+\left( \mathbb E \left[\left|   F (v(\tau_i, \lambda_j)) \Delta_{i,j}(V)\right|^2 \right]\right)^{\frac12} \\
   \lesssim & \,  N^{-\frac12-H}.
  \end{align*}
 Combining the two bounds above and summing over  $i, j=0, \cdots, N-1$, we conclude that 
 \begin{align*}
\mathbb E [I_1] \lesssim  N^{-H}.
  \end{align*} 

\noindent\textbf{Step 2:  Estimation of  $I_2$.}
      Let $ \mathcal F_t$ be the   $\sigma$-algebra defined in \eqref{eq sigma filed}.  Consider any pairs $(i, j), (i', j')\in \mathbb N^2$  such that  $i+j\ge i'+j'+2$. By \eqref{eq transform} and  the temporal independence of  $\{W(t, x)\}_{t\ge0, x\in\mathbb R}$,    the random variables    $v(\tau_{i'}, \lambda_{j'}), v(\tau_i, \lambda_j)$,  and $\Delta_{i',j'}(V)$ are  $\mathcal F_{\sqrt{2}(\tau_i+\lambda_j)}$-measurable,   while  $\Delta_{i,j}(V)$   is independent of   $\mathcal F_{\sqrt{2}(\tau_i+\lambda_j)}$. Taking the conditional expectation with respect to    $\mathcal F_{\sqrt{2}(\tau_i+\lambda_j)}$, we obtain
    \begin{align*}
  & \mathbb E\left[   F^2(v(\tau_i, \lambda_j) )  F^2(v(\tau_{i'}, \lambda_{j'})) \left( N^{2H-1} \Delta_{i,j}(V)^2   -2^{H-\frac{5}{2}}N^{-2}\right)   \left( N^{2H-1}  \Delta_{i',j'}(V)^2   -2^{H-\frac{5}{2}}N^{-2}\right) \right]\\
   =&\, \mathbb E\Bigg[   F^2(v(\tau_i, \lambda_j) )  F^2(v(\tau_{i'}, \lambda_{j'}))     \left( N^{2H-1}  \Delta_{i',j'}(V)^2   - 2^{H-\frac{5}{2}}N^{-2}\right) \\
   &   \ \ \ \ \ \  \ \ \ \ \     \cdot  \mathbb E\left[  \left( N^{2H-1} \Delta_{i,j}(V)^2   -2^{H-\frac{5}{2}}N^{-2}\right)|\mathcal F_{\sqrt{2}(\tau_i+\lambda_j)}\right] \Bigg]  \\
   =& \, 0,
     \end{align*}
     where    the last equality follows from   \eqref{Eq:rec-inc}.
     
For pairs with   $(i, j), (i', j')\in \mathbb N^2$  with $|(i+j)-( i'+j')|<2$,  we apply      the  Cauchy-Schwarz  inequality and the moment estimate  \eqref{2f-2}  to get
       \begin{align*} 
& \mathbb E\left[   F^2(v(\tau_i, \lambda_j)   F^2(v(\tau_{i'}, \lambda_{j'})  \left( N^{2H-1} \Delta_{i,j}(V)^2   - 2^{H-\frac{5}{2}}N^{-2}\right) \left( N^{2H-1} \Delta_{i',j'}(V)^2   -2^{H-\frac{5}{2}}N^{-2}\right)    \right]\\
    \lesssim & \,   \left( \mathbb E\left[   F^8(v(\tau_i, \lambda_j) \right]\right)^{1/4 }  \left( \mathbb E\left[   F^8(v(\tau_{i'}, \lambda_{j'}) \right]\right)^{1/4 }  \\
 & \cdot   		\left(	\mathbb{E}\left[  \left( N^{2H-1} \left( \Delta_{i,j}(V) \right)^{2} -2^{H-\frac{5}{2}}N^{-2}\right) ^{4}\right] \right)^{1/4 } 
  \left(	\mathbb{E}\left[  \left( N^{2H-1} \left( \Delta_{i',j'}(V) \right)^{2} -2^{H-\frac{5}{2}}N^{-2}\right) ^{4}\right] \right)^{1/4 } \\
 \lesssim &\,    N^{-4}. 
      \end{align*}
Summing over all such pairs yields 
 \begin{align*}
\mathbb E [|I_2|]\le c N^{-1}.
  \end{align*} 
  
  \noindent\textbf{Step 3:  Estimation of  $I_3$.}  Using the H\"older continuity of $v$  from \eqref{eq Holder v},  the Lipschitz continuity of $F$, and the Cauchy-Schwarz inequality, we  obtain
        \begin{align*}
  \mathbb E[|I_3|] \le &  \,  2^{H-\frac{5}{2}} \sum_{i=0}^{N-1}\sum_{j=0}^{N-1} \int_{\tau_i}^{\tau_{i+1}}  \int_{\tau_j}^{\tau_{j+1}}    \mathbb E  \left[| F^2(v(\tau_i, \lambda_j)) - F^2(v(\tau, \lambda))|  \right]  d\tau d\lambda\\
   \lesssim &  \,   \sum_{i=0}^{N-1}\sum_{j=0}^{N-1} \int_{\tau_i}^{\tau_{i+1}}  \int_{\tau_j}^{\tau_{j+1}}  \left(  \mathbb E  \left[| F(v(\tau_i, \lambda_j)) - F(v(\tau, \lambda))| ^2 \right] \right)^{\frac12} d\tau d\lambda\\
  \lesssim &\,    N^{-H}.
    \end{align*} 
       The proof is complete. 
    \end{proof}
  
 \subsection{Numerical experiments}
   
  Let us consider the 
equation  \eqref{SWE para} with $F(u)=1+\sin(u)$ and $\theta=2$.  We simulate the paths of the solution \eqref{SWE para}, using  the R procedure \textbf{fieldsim} which is available via \cite{BIL07}. The following  numerical results are consistent with the theoretical results from Corollary \ref{coro est}.

 \subsubsection{Space-time white noise}
 Take the parameter  $H=0.5$.   Using  Corollary \ref{coro est},  we compute   the estimator $\hat \theta_N$ defined in  \eqref{eq appr2}, and consequently their sample mean and sample standard deviation. In Figure  \ref{fig:left1}, we present one realization of the estimator $\hat \theta_N$ for $N=20, \cdots, 60$, and  we display the absolute and relative errors in Figure  \ref{fig:right1}.

\begin{figure}[htbp]
    \centering
    \begin{minipage}[t]{0.48\textwidth}
        \centering
        \includegraphics[width=\linewidth]{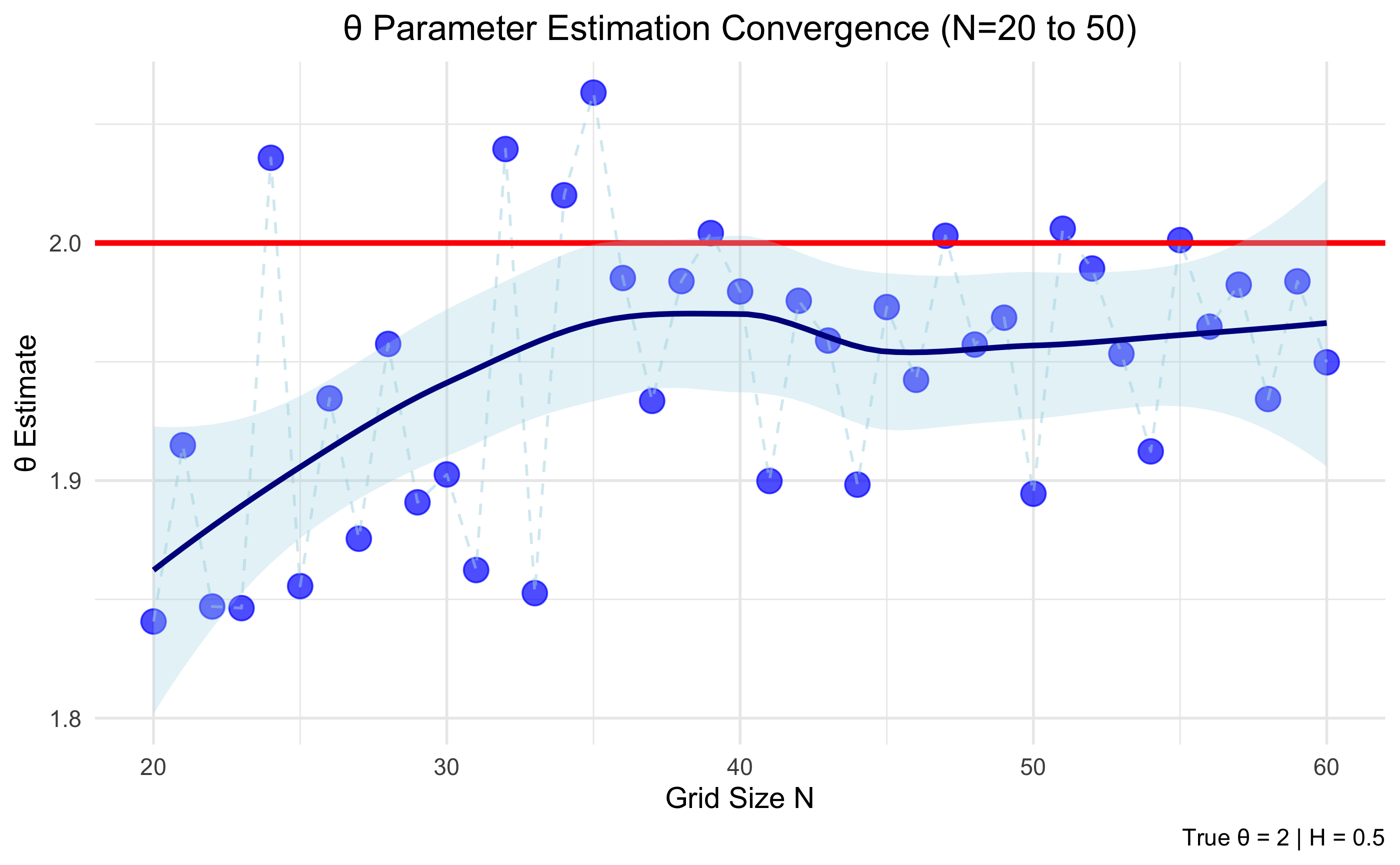}
       \caption{ }
        \label{fig:left1}
    \end{minipage}
    \hfill
    \begin{minipage}[t]{0.48\textwidth}
        \centering
        \includegraphics[width=\linewidth]{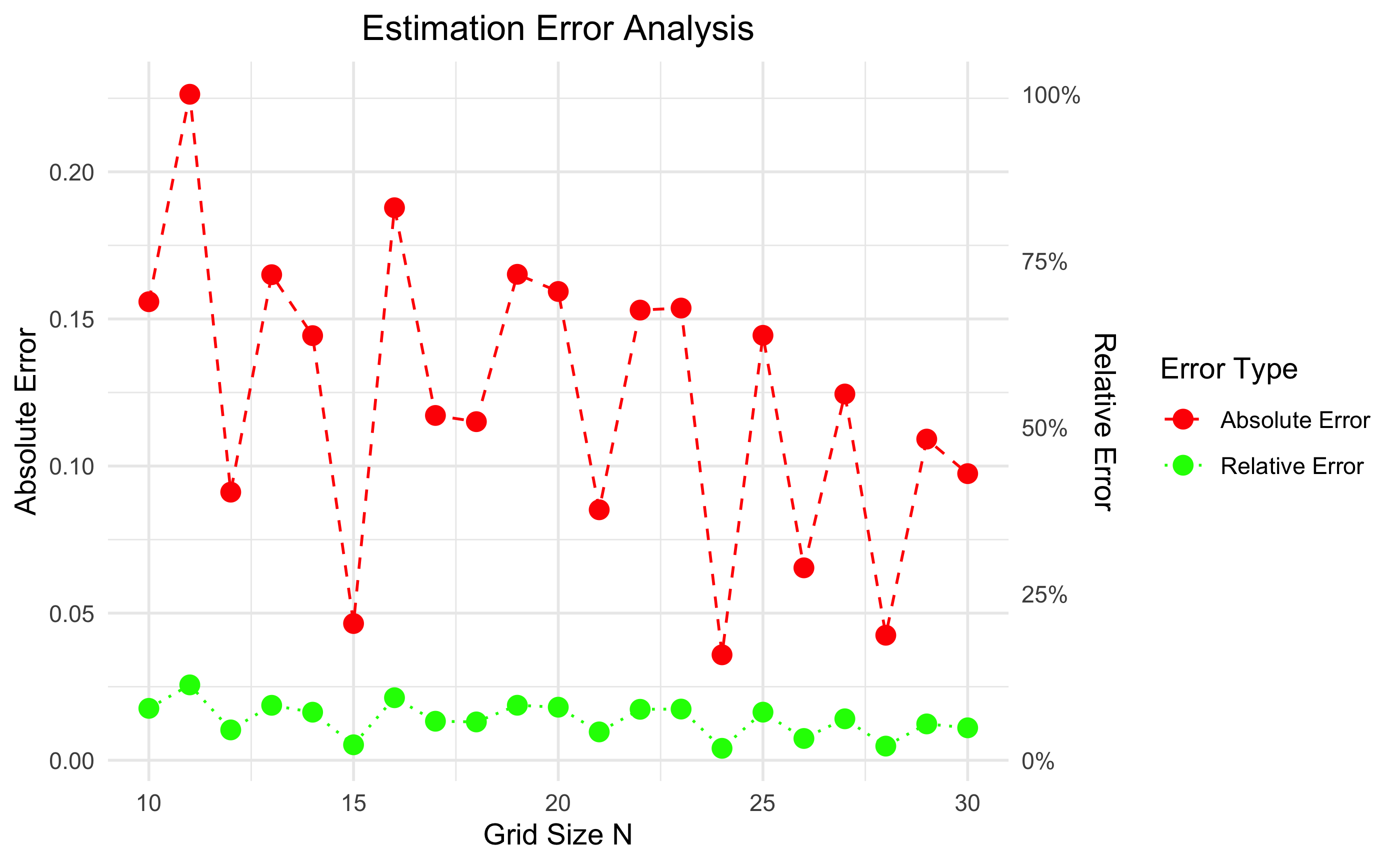}
         \caption{ }
        \label{fig:right1}
    \end{minipage}
\end{figure}
       
 \subsubsection{Time-white and space-colored noise} Take the parameter  $H=0.55$. 
   Using  Corollary \ref{coro est},  we compute   the estimator $\hat \theta_N$ defined in  \eqref{eq appr2}, and consequently their sample mean and sample standard deviation. In Figure  \ref{fig:left2}, we present one realization of the estimator $\hat \theta_N$ for $N=30, \cdots, 60$, and  we display the absolute and relative errors in Figure  \ref{fig:right2}.
     
\begin{figure}[htbp]
    \centering
    \begin{minipage}[t]{0.48\textwidth}
        \centering
        \includegraphics[width=\linewidth]{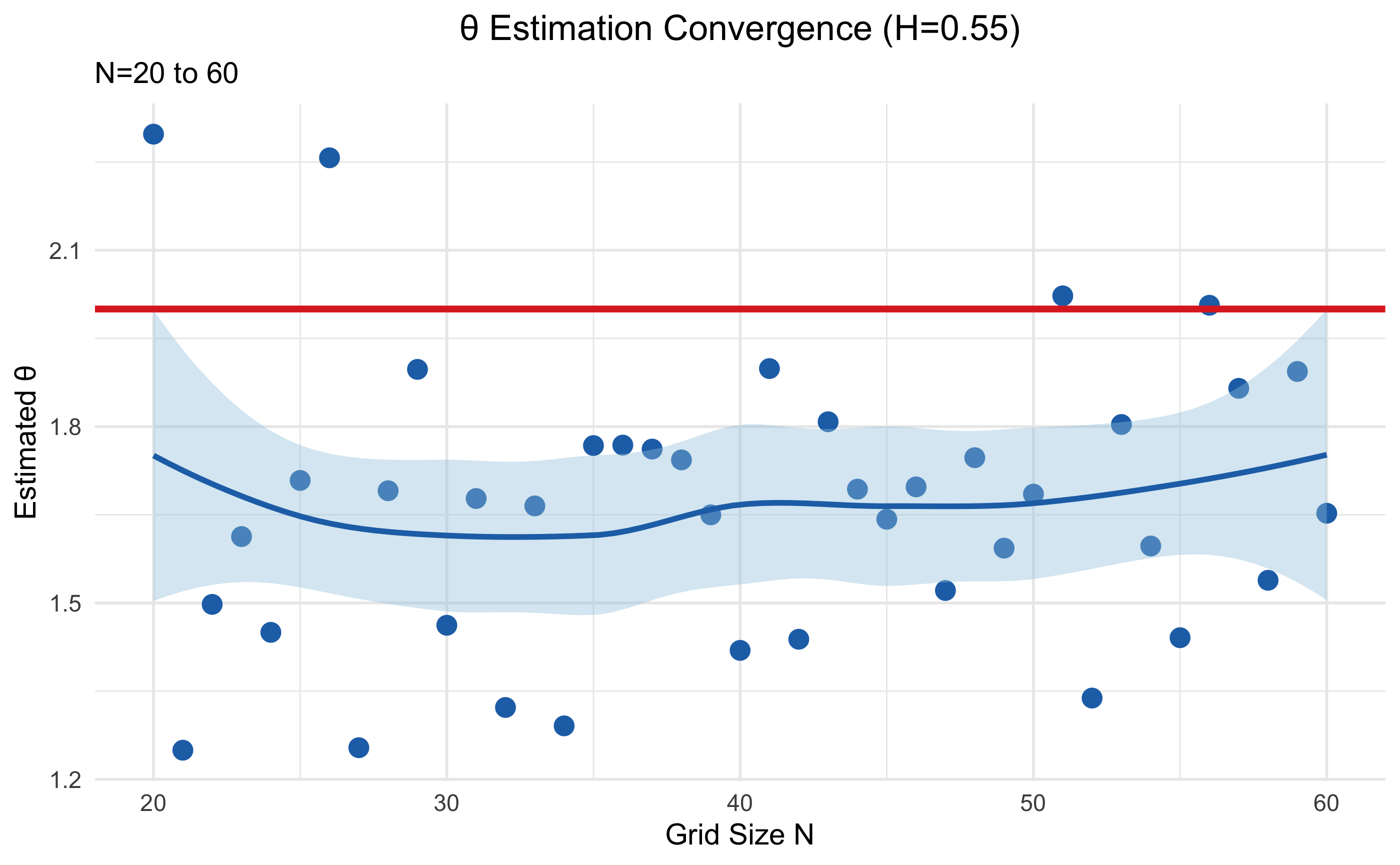}
       \caption{}
        \label{fig:left2}
    \end{minipage}
    \hfill
    \begin{minipage}[t]{0.48\textwidth}
        \centering
        \includegraphics[width=\linewidth]{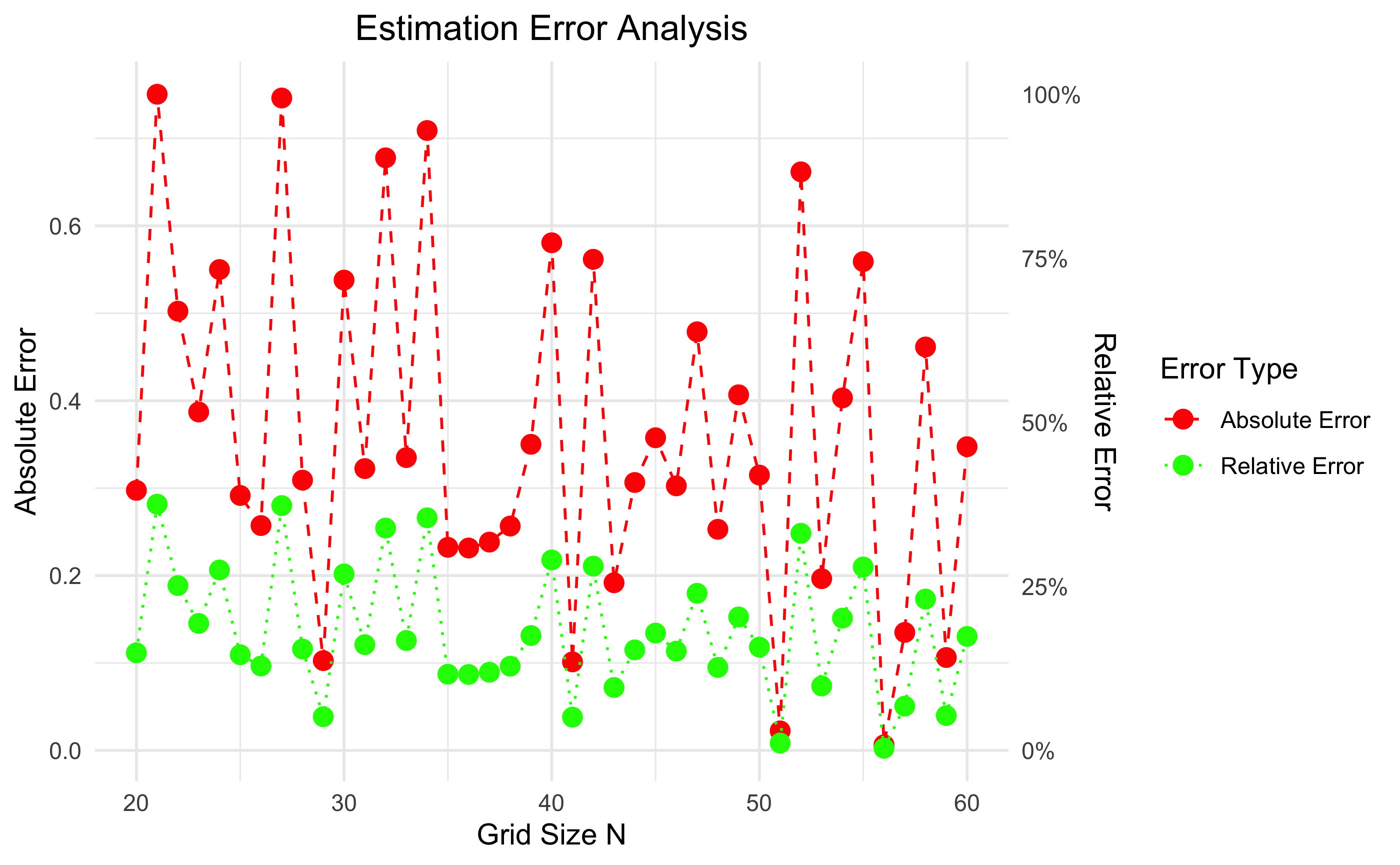}
        \caption{}
        \label{fig:right2}
    \end{minipage}
\end{figure}

 \vskip0.3cm 
 
\noindent{\bf Acknowledgments}  The research of G. Liu is  supported by the NSFC (No. 11801196). The research of R. Wang is partially supported by the NSF of Hubei Province (No. 2024AFB683) and  the Wuhan University Social Science Digital Innovation Research Team Project (No. WDSZTD2024B05).

 \vskip0.3cm 
 
 \noindent{\bf Author Contributions}  G. Liu  and R. Wang designed the inference methodology, implemented the method, conducted the simulation studies and data analyses, and drafted  the manuscript. All authors reviewed the manuscript.
\vskip0.3cm

\noindent{\bf {\large Declarations}}

\vskip0.3cm
\noindent{\bf Conflict of interest}  No potential conflict of interest was reported by the authors.

\vskip0.8cm

    \end{document}